\documentclass[11pt,a4paper]{amsart}

\usepackage{amssymb, amsmath, amsthm}
\usepackage{enumerate}
   
\theoremstyle{plain} \newtheorem{theorem}{Theorem}[section]
\theoremstyle{plain} \newtheorem{corollary}[theorem]{Corollary}
\theoremstyle{plain} \newtheorem{proposition}[theorem]{Proposition}
\theoremstyle{plain}\newtheorem{lemma}[theorem]{Lemma}
\theoremstyle{definition} \newtheorem{definition}[theorem]{Definition}
\theoremstyle{definition}
\theoremstyle{remark}\newtheorem{remark}[theorem]{Remark}
\theoremstyle{remark}\newtheorem{remarks}[theorem]{Remarks}

\newcommand{\C}{{\mathbb{C}}}
\newcommand{\R}{{\mathbb{R}}}
\newcommand{\F}{{\mathcal{F}}}
\newcommand{\X}{\mathfrak{X}}
\newcommand{\D}{\mathcal{D}}
\newcommand{\DF}{\mathcal{D}_{\mathcal F}}
\newcommand{\id}{\operatorname{Id}}
\newcommand{\AF}{\operatorname{Aut}(\mathcal{F})}
\newcommand{\aF}{\operatorname{\mathfrak{aut}}(\mathcal{F})}
\newcommand{\G}{\mathcal G}
\newcommand{\A}{\mathcal A}

\usepackage[all]{xy}

\title[On the automorphism group of foliations]{On the automorphism group of foliations with geometric transverse structure}
\author{Laurent Meersseman, Marcel Nicolau and Javier Rib\'{o}n}
\date{}

\thanks{This work was partially supported by the grant MTM2015-66165-P from the Ministerio de Economia y Competitividad  of Spain.}

\subjclass{58D05, 53C12, 22E65}

\address{Laurent Meersseman\\ LAREMA\\ Universit\'{e} d'Angers\\ 
F-49045 Angers, France\\ meersseman@univ-angers.fr}

\address{Marcel Nicolau\\ Departament de Matem\`{a}tiques \\ 
Universitat Aut\`{o}noma de Barcelona \\ E-08193  Cerdanyola, Spain \\
nicolau@mat.uab.cat} 

\address{Javier Rib\'{o}n\\ Instituto de Matem\'{a}tica \\
Universidade Federal Fluminense\\ 
Niter\'{o}i, Rio de Janeiro, Brasil \\ jribon@id.uff.br}

\begin{document}

\begin{abstract}
Motivated by questions of deformations/moduli in foliation theory, we investigate the structure of some 
groups of diffeomorphisms preserving a foliation. We give an example of a $C^\infty$ foliation whose 
diffeomorphism group is not a Lie group in any reasonable sense. On the positive side, we prove that 
the automorphism group of a transversely holomorphic foliation or a riemannian foliation is a strong 
ILH Lie goup in the sense of Omori.
\end{abstract}

\maketitle



\section{Introduction}

Let $M$ be a closed smooth manifold.
We denote by
$\D= \D(M)$ the group of all $C^\infty$-diffeomorphisms of $M$ endowed with the $C^\infty$ topology. 
It is a metrizable topological group. 
In \cite{Les1}, Leslie proved that $\D$ is a Fr\'{e}chet Lie group; namely, 
a group with a structure of infinite dimensional manifold modeled on a Fr\'{e}chet space 
such that the group operations are smooth.  
Fr\'{e}chet manifolds are difficult to deal with as many of the fundamental results of calculus, 
like implicit function theorem 
or Frobenius theorem, are not valid in that category.  However, Palais proved that $\D$ has a richer structure: 
it is a ILH Lie group, that is a topological group that is
the inverse limit of Hilbert manifolds $\D^k$ each of which is a topological group \cite{Pal, Om1}. 
Later on, Omori introduced the notion of {\sl strong} ILH Lie group, whose precise 
definition is recalled in~\ref{ILH}. The interest of this stronger notion is that it allows to
formulate an implicit function theorem and 
a Frobenius theorem in that setting (cf. \cite{Om3}). In addition to demonstrate those results, Omori used them 
to show that $\D$ is a strong ILH Lie group.

It is natural to ask for similar results for groups of diffeomorphisms preserving a geometric structure. In \cite{EM}, 
Ebin and Marsden proved the existence of a structure of a ILH Lie group on 
$\D_\eta$ and $\D_\omega$,
the subgroups of smooth diffeomorphisms preserving a volume element $\eta$ or a 
symplectic form $\omega$ on $M$ respectively; and Omory showed that some subgroups of $\D$, including  $\D_\eta$ and $\D_\omega$, 
are strong ILH Lie groups (cf. \cite{Om3, Om4}). 
\vspace{5pt}\\
\indent In this paper, we are interested in the structure of groups of diffeomorphisms preserving a foliation; and preserving 
a foliation with some additional geometric transverse structure. It turns out that the situation is much more complicated for 
those subgroups. We give in this paper a negative result but also two positive ones.
\vspace{5pt}\\
\indent Our motivation comes from deformation problems. We want to understand the deformations/moduli space of foliations 
with a geometric transverse structure. The existence of a "nice" moduli space is closely related to the existence of a "nice" 
Lie group structure on the associated diffeomorphism groups (see for example \cite{MN}). A strong ILH Lie group structure 
is a nice structure, especially because of the implict function and Frobenius Theorem proved by Omori in this context.

Let us first assume that $\F$ is a smooth foliation on $M$.  
Call $\DF$ the group of foliation preserving diffeomorphisms.
In \cite{Les2}, Leslie asserted that $\D_\F$ is a Fr\'{e}chet Lie group. Unfortunately there was a gap 
in his demonstration, which remains only valid for foliations with
a transverse invariant connection. 
Indeed, for certain foliations the group $\D_\F$ is not a Lie group in any reasonable sense.
This fact seems to be folklore although, as far as we know, there is no published example. 
In part \ref{cex}, we construct and discuss in detail such a foliation.

Let us now consider a type of geometric structure $\Xi$ having 
the property that if $\epsilon$ is a structure of type $\Xi$ on a compact manifold $M$ then the automorphism 
group $\operatorname{Aut}(M,\epsilon)$ is a Lie group in a natural way. Assume that the foliation $\F$ 
is endowed with a geometric transverse structure of type $\Xi$. Because of the previous negative result, there is no reason for the automorphism
group $\AF$ (that is the group of the diffeomorphisms of $M$ which preserve the foliation as well as 
its geometric transverse structure) to be a strong ILH Lie group. We show this is nevertheless true in the following two cases: $\F$ is 
a transversely holomorphic foliation (part \ref{TH}) or a Riemannian foliation (part \ref{RF}).

\section{groups of diffeomorphisms preserving a foliation}
\label{cex}
Throughout this article $M$ will be a fixed closed smooth manifold. Let $\X=\X(M)$ be the Lie algebra of smooth (of class $C^\infty$)
vector fields on $M$. Endowed with the $C^\infty$-topology, $\X$ is a separable Fr\'{e}chet space. We denote by
$\D= \D(M)$ the group of smooth (of class $C^\infty$) diffeomorphisms of $M$ endowed with the $C^\infty$ topology. 
It is a topological group, which is metrizable and separable; so it has second countable topology.


We recall that a Fr\'{e}chet manifold is a Hausdorff topological space with an atlas 
of coordinate charts taking values in Fr\'{e}chet  spaces in such a way that the coordinate 
changes are smooth maps  (i.e. they admit continuous 
Fr\'{e}chet derivatives of all orders). A Fr\'{e}chet Lie group is a Fr\'{e}chet manifold
endowed with structure of group such that multiplication and inverse maps are smooth. 
In \cite{Les1}, it was proved by Leslie that $\D$ is a Fr\'{e}chet Lie group with Lie algebra $\X$. 
In fact, $\D$ is endowed with a richer structure: it is the inverse limit of Hilbert manifolds $\D^k$ 
each of them being a topological group \cite{EM, Om1}. Moreover $\D$ has in fact 
the structure of a {\sl strong} ILH-Lie group \cite{Om3}, a notion introduced by Omori 
that can be defined as follows.

Recall that a Sobolev chain is a system $\{\mathbb E, E^k, k\in \mathbb N\}$ where $E^k$ are Hilbert spaces, with $E^{k+1}$
linearly and densely embedded in $E^k$, and $\mathbb E = \bigcap E^k$ has the inverse limit topology.

\begin{definition}\label{ILH}
A topological group $G$ is said to be a strong ILH-Lie group modeled on a Sobolev chain $\{\mathbb E, E^k, k\in \mathbb N\}$
if there is a family
$\{G^k, k\in \mathbb N\}$ 
fulfilling
\begin{enumerate}[1)]
\item $G^k$ is a smooth Hilbert manifold, modeled on  $E^k$, and a topological group; moreover 
$G^{k+1}$ is embedded as a dense subgroup of $G^k$ with 
inclusion of class $C^\infty$.
\item $G=\bigcap G^k$ has the inverse limit topology and group structure.
\item Right translations $R_g\colon G^k\to G^k$ are $C^\infty$.
\item Multiplication and inversion extend to $C^m$ mappings $G^{k+m}\times G^k \to G^k$ and 
$G^{k+m}\to G^k$ respectively.
\item Let $\mathfrak g^k$ be the tangent space of $G^k$ at the neutral element $e$ and let $T{G^k}$ be the tangent bundle of $G^k$. 
The mapping $dR\colon \mathfrak g^{k+m} \times G^k \to T{G^k}$, defined by $dR(u , g)= dR_g u$,
is of class $C^m$.
\item 
There is a local chart $\xi\colon \tilde U\subset \mathfrak g^1 \to G^1$, with $\xi(0) = e$ whose restriction
to $\tilde U\cap \mathfrak g^k$ gives a local chart of $G^k$ for all $k$.
\end{enumerate}
\end{definition}

\begin{remarks}
\begin{enumerate}[a)]
\item The strong ILH-Lie group $\D$ is modeled on the Sobolev chain 
$\{\X, \X^k, k\in \mathbb N\}$, where $\X^k$ are the Sobolev completions of $\X$. In that case $G^k= \D^k$ are the Sobolev 
completions of $\D$ and there are natural identifications $\mathfrak g\equiv \X$ and $\mathfrak g^k\equiv \X^k$
(cf. \cite{Om3}, Theorem 2.1.5).

\item Weaker definitions of ILH-Lie groups  were considered in articles prior to  \cite{Om3}, in particular in 
\cite{EM, Om1}. In the present paper we will only consider ILH-Lie groups in the strong sense given by the above definition.

\item This strong definition ensures, by means of the last condition, that $G$ is endowed with a structure of Fr\'{e}chet Lie group.
The space $\mathfrak g = \cap \mathfrak g^k$ with the inverse limit topology is a Fr\'{e}chet space. Moreover, it is naturally 
endowed with a Lie algebra structure (cf. \cite{Om3}, I.3). We call $\mathfrak g$ the Lie algebra of $G$. 
\end{enumerate}
\end{remarks}

Suppose that $H$ is a subgroup of a strong ILH-Lie group $G$ and let $\mathfrak B$ be a basis of neighborhoods of $e$
in $G$. For each $U\in \mathfrak B$, let $U_0(H)$ be the set of points $x\in U\cap H$ such that $x$ and $e$
can be joined by a piecewise smooth curve $c(t)$ in  $U\cap H$. Here, piecewise smooth means that 
the mapping $c\colon [0,1] \to G^k$ is piecewise of class $C^1$ for each $k$. Then the family 
$\mathfrak B_0(H)= \{U_0(H)\mid U\in\mathfrak B\}$ satisfies the axioms of neighborhoods of 
$e$ of topological groups. As in \cite{Om3},
this topology will be called the LPSAC-topology of $H$, where LPSAC stands for ``linear piecewise 
smooth arc-connected". In general the LPSAC-topology does not coincide with the induced topology.

\begin{definition}
Let $G$ be a strong ILH-Lie group 
and let $\mathfrak g^k$ be the tangent space of $G^k$ at $e$. 
A subgroup $H$ of $G$ is called a {\sl strong ILH-Lie subgroup} of $G$ if 
there is a splitting $\mathfrak g = \mathfrak h \oplus \mathfrak f$ such that
the following 
two conditions are fulfilled:
\begin{enumerate}[i)]
\item It induces splittings  $\mathfrak g^k = \mathfrak h^k \oplus \mathfrak f^k$ for each $k$, where 
$\mathfrak h^k$ and $\mathfrak f^k$ are, respectively, the closures of $\mathfrak h$ and $\mathfrak f$ in $\mathfrak g^k$.
\item There is a local chart $\xi\colon \tilde U\subset \mathfrak g^1 \to G^1$ fulfilling condition~6) in Definition~\ref{ILH}
such that $\xi(\tilde U\cap\mathfrak h) =U_0(H)$, where $U = \xi(\tilde U \cap \mathfrak g)$
\end{enumerate}
\end{definition}

\begin{remark}\label{countable}
Clearly, a strong ILH-Lie subgroup $H$ of a strong ILH-Lie group $G$ is itself a strong 
ILH-Lie group under the LPSAC-topology. 
\end{remark}

In \cite{Om3}, Omori was able to state some of the classical theorems of analysis, 
like the implicit function theorem or Frobenius theorem, in the setting of strong ILH-Lie groups.
The following is a special case of Frobenius theorem that will be used later 
(cf. \cite{Om3}, Theorem 7.1.1 and Corollary 7.1.2).

\begin{theorem}[Omori]\label{Frobenius}
Let $E$ and $F$ be Riemannian vector bundles over $M$ and let $A\colon \Gamma(TM)\to \Gamma(E)$
and $B\colon \Gamma(E)\to \Gamma(F)$ be differential operators of order $r\geq 0$ with smooth coefficients.
Assume that the following conditions are fulfilled:
\begin{enumerate}[i)]
\item $B\,A = 0$.
\item $A\, A^\ast + B^\ast\,B$ is an elliptic differential operator, where $A^\ast$, $B^\ast$ are
formal adjoints of $A$, $B$ respectively.
\item $\mathfrak h = \ker A$ is a Lie subalgebra of $\X=\Gamma(TM)$.
\end{enumerate}
Then there is a strong ILH-Lie subgroup $H$ of $\D$ whose Lie algebra is $\mathfrak h$.
Moreover, $H$ is an integral submanifold of the distribution $\mathcal H$ in $T\D$ given by 
right translation of $\mathfrak h$, i.e. $\mathcal H = \{dR_f\mathfrak h \mid f\in \D\}$.
\end{theorem}

Here, the notation $\Gamma(E)$ stands for the space of smooth sections of class $C^\infty$
of a given vector bundle $E$. 

As already remarked, the group $H$ in the above theorem is a 
strong ILH-Lie group with the LPSAC-topology. 
However that topology on $H$ can be stronger that the topology induced by $G$,
even if $H$ is closed in $G$. Therefore the above theorem is not sufficient by itself to assure that 
$H$, endowed with the induced topology, is a strong ILH-Lie group, or even a Fr\'{e}chet Lie group. 
The following Proposition gives sufficient conditions for the above two topologies 
to coincide.
(cf. \cite{Om3},~VII.1).

\begin{proposition}[Omori]\label{two-top} 
Let $H$ be a strong ILH-Lie subgroup of $G$ satisfying the conditions of the above theorem and assume that 
the following conditions are also fulfilled
\begin{enumerate}[\ \ a)]
\item $H$ is closed in $G$, 
\item $G$ is second countable and the LPSAC-topology of $H$ is also second countable.
\end{enumerate}
Then the topology on $H$ induced by $G$ coincides with the LPSAC-topology. 
\end{proposition}

\medskip
From now on we assume that $M$ is endowed with a smooth foliation $\F$ of dimension $p$ 
and codimension $m$.
We denote by $T\F$ its tangent bundle and by $\nu\F= TM/T\F$ the normal 
bundle of $\F$. 
The foliation $\F$ decomposes the manifold into the 
disjoint union of $p$-dimensional
submanifolds that are called the leaves of the foliation. 

An atlas adapted to $\F$ is a smooth atlas $\{(U_i, \varphi_i)\}$ of $M$,  
with $\varphi_i \colon U_i \to \mathbb R^p\times\R^m$ homeomorphisms, such that
the leaves $L$ of $\F$ are defined on $U_i$ by the level sets 
$\phi_i= constant$ of the submersions $\phi_i = pr_2\circ \varphi_i$.
Here $pr_2\colon\mathbb R^p\times\R^m\to \R^m$ denotes the projection onto the second factor.
In that case there is a cocycle $\{\gamma_{ij}\}$ of local smooth transformations of $\R^m$ 
such that 
\begin{equation}\label{cocycle}
\phi_j = \gamma_{ji}\circ \phi_i.
\end{equation} 
The subsets $\phi_i= constant$
are called the slices of $\F$ in $U_i$.

We consider now the following subgroups of $\D$ of foliation preserving diffeomorphisms
$$
\begin{array}{rcl}
\D_\F & = & \{f\in \D \mid f_\ast \F = \F\}, \\[2mm]
\D_{L} & = & \{f\in \D_\F \mid f (L) = L,\;\; \text{for each leaf} \; L\}.
\end{array}
$$
We denote by $\X_\F$ the Lie algebra of foliated vector fields 
(i.e. of vector fields whose flows preserve the foliation) 
and by $\X_{L}$ 
the Lie algebra of vector fields
that are tangent to $\F$ (i.e. vector fields that are smooth sections of $T\F$).
Notice that the flows associated to vector fields of $\X_\F$ and $\X_{L}$
define one-parameter subgroups of $\D_\F$ and $\D_{L}$ respectively.
The following result was proved by Omori in \cite{Om3}. We include a sketch of proof for later use.

\begin{proposition}[Omori]\label{d-tangent}
The group $\D_{L}$ is a strong ILH-Lie subgroup of 
$\D$ with Lie algebra $\X_{L}= \Gamma(T\F)$.
\end{proposition}

\begin{proof}
Let $\pi\colon \Gamma(TM) \to \Gamma(\nu\F)$ be the linear mapping induced by the natural projection
$TM\to \nu\F$. We can regard $\pi$ as differential operator of order $0$ and its adjoint operator $\pi^\ast$,
constructed using Riemannian metrics on $TM$ and $\nu\F$, is also a differential operator of order $0$
such that $\pi\pi^\ast$ is elliptic (i.e. an isomorphism). Theorem~\ref{Frobenius} implies now that there is a
strong ILH-Lie subgroup $\D'_{L}$ of $\D$ with Lie algebra $\X_{L}$. Let $\D^k_{L}$ and $\D'^k_{L}$ be the 
Sobolev completions of $\D_{L}$ and $\D'_{L}$ respectively. Clearly $\D'_{L}\subset \D_{L}$ and 
$\D'^k_{L}\subset \D^k_{L}$. Since $\D'_{L}$ is obtained by the Frobenius theorem, if a piecewise $C^1$-curve
$c(t)$ satisfies $c(0) = e$ and $c(t)\in \D^k_{L}$, then $c(t)\in \D'^k_{L}$. 
%
%
Using the exponential 
mapping with respect to a connection under which $\F$ is parallel, one can see that 
$\D^k_{L}$ 
has a structure of Hilbert manifold, hence  it is LPSAC. This implies that 
$\D'^k_{L}= \D^k_{L,0}$ (the connected component of $\D^k_{L}$ containing $e$)
and therefore that $\D'_{L}$ is the connected component 
of $\D_{L}$ containing the identity. 
\end{proof}

\begin{remark}
It follows from the above Proposition that $\D_{L}$ is a 
strong ILH-Lie group with the LPSAC-topology and, in particular, it is endowed with a structure of Fr\'{e}chet Lie group.
In general, however, $\D_{L}$ is not closed in $\D$.  This is the case for instance if $\F$ is a linear foliation 
on the 2-torus $T^2$ with irrational slope.
\end{remark}

In \cite{Les2}, Leslie asserted that the group $\D_\F$ is a Fr\'{e}chet Lie group. Unfortunately there was an error 
in his demonstration, which remains valid only for Riemann's foliations or, more generally, for foliations with
a transverse invariant connection. 
In fact, the group $\D_\F$ may not be a Lie group in any reasonable sense.
That fact seems to be folklore although, as far as we know, there is no published example. 
We exhibit here a concrete one. More precisely we prove

\begin{theorem}\label{contra}
There is a foliation $\F$ on the $2$-torus $T^2$ such that $\D_{\F}$ cannot 
be endowed with a structure of a topological group fulfilling 
\begin{enumerate}[\ i)]
\item The inclusion $\D_\F \hookrightarrow \D(T^2)$ is continuous.
\item $\D_\F$ has second countable topology.
\item $\D_\F$ is locally path-connected.
\end{enumerate}
\end{theorem}

We dedicate the rest of the section to prove that statement.

Assume that an orientation preserving diffeomorphism $h \in \D(S^1)$ has been given. 
Let $\tau\colon \R\rightarrow\R$ be the translation $\tau(t)= t+1$ and
let $\tilde h\colon \R \times S^1 \rightarrow \R \times S^1$ denote the map defined by
$$
\tilde h(t, x) = (\tau(t) , h(x)) = (t + 1, h(x)).
$$
Since $h$ is isotopic to the identity, the quotient manifold $M=\R \times S^1/\langle \tilde h\rangle$ is
the 2-torus $T^2$.
The foliation $\tilde{\F}$ of $\R \times S^1$ whose leaves are $\R\times \{x\}$ is preserved by $\tilde h$. 
Hence, $\tilde{\F}$ induces a well defined foliation $\F$ on $M$, which is transverse to the fibres of the natural projection
$$
\pi \colon \R \times S^1 /\langle \tilde h\rangle \,\longrightarrow S^1= \R/\langle \tau \rangle.
$$
We say that $\F$ is the foliation on $T^2$ obtained by suspension of $h$.

We fix a circle $C = \pi^{-1}(t_0)$ inside the 2-torus $T^2$. In a neighborhood $U$ of $C$, the foliation $\F$ is 
a product $(t_0-\epsilon, t_0+\epsilon) \times C$ and each element $f\in \DF$ close enough to the identity map $\id$ 
sends $C=\{t_0\} \times C$ into $U$. Moreover, in the adapted
coordinates $(t, x)$, it is written
$$
f(t,x) = (f_1(t, x), f_2(x)).
$$
Therefore, there is a neighborhood $\mathcal W$ of $\id$ in $\DF$ such that the map
\begin{equation}\label{P}
\begin{array}{rcc}
P \colon \; \mathcal W & \rightarrow & \D(S^1) \\
  f & \mapsto & f_2
\end{array}
\end{equation}
is well-defined. Notice that, if we consider on $\D(S^1)$ the topology induced by $\D$, then $P$ is a continuous map; 
it is also a morphism of local groups.

\begin{lemma}
The map $P$ sends $\mathcal W$ into the centralizer $Z^\infty(h)$ of $h$ in $\D(S^1)$.
Moreover, $P$ has a continuous right-inverse map $\sigma$ defined on a neighborhood $\mathcal V$ 
of the identity in $Z^\infty(h)$.
\end{lemma}

\begin{remark}
It follows from the above statement that the image of the map $P$
contains a neighborhood of the identity in $Z^\infty(h)$.
\end{remark}

\begin{proof}
The first assertion follows from an easy computation.
%
Given an element $g_1\in Z^\infty(h)$ close enough to the identity, the map
$\tilde g\colon \R \times S^1 \rightarrow \R \times S^1$ defined by $\tilde{g}(t,x)=(t,g_{1}(x))$
commutes with $\tilde h$ and preserves the foliation $\F$. Hence, it induces an element $g\in \DF$
which is close to the identity. The map $g_1\mapsto \sigma(g_1) = g$ is well-defined on a 
neighborhood $\mathcal V$ of the identity in $Z^\infty(h)$. Clearly, the map $\sigma$ is continuous and fulfills 
$P\circ \sigma = \id$. 
\end{proof}

\begin{lemma}\label{restriction}
There is a neighborhood $\mathcal W'$ of the identity in $\D(T^2)$ and a smooth map 
$\Psi\colon \mathcal W' \to \D(S^1)$ such that $P$ is the composition
$$
P\colon \mathcal W'\cap \DF \hookrightarrow \mathcal W' \overset{\Psi}\longrightarrow \D(S^1)
$$
\end{lemma}

\begin{proof}
Let  $\mathcal W'$ be a neighborhood of the identity in $\D(T^2)$ such that $P$ is defined
on $\mathcal W'\cap \DF$. If $\mathcal W'$ is small enough, we can also assume that,
for each $f\in \mathcal W'$, 
$f(\{t_0\}\times C)$ is contained in $U$ and it is transverse to the foliation $\F$. We define 
$\Psi(f)$ as the map 
$$\Psi(f) =  \operatorname{pr}_2 \circ f|_C \colon C\longrightarrow C$$ 
where 
$ \operatorname{pr}_2\colon U=(t_0-\epsilon, t_0+\epsilon)\times C \to C$ is the projection
onto the second factor. We can think of $\Psi(f)$ as an element of $\D(S^1)$. Clearly the map
$\Psi$ is smooth and fulfills the required conditions.
\end{proof}


Denote by $\D^r(S^1)$ the group of diffeomorphisms of $S^1$ of class $r$, where $0\leq r \leq \infty$,
endowed with the $C^r$-topology. 
For a given $h\in \D(S^1) = \D^\infty(S^1)$,
we denote by $Z^r(h)$ the centralizer of $h$ in $\D^r(S^1)$, and we denote by
$Z^r_0(h)$ the closure in $\D^r(S^1)$ of the group $\langle h \rangle$
generated by $h$. Notice that $Z^r_0(h)\subset Z^r(h)$.

Let
$$\rho\colon \D^0(S^1)\rightarrow S^1$$ 
be the map that associates to each element $h\in  \D^0(S^1)$
its rotation number $\rho(h)$.
We assume that  $h\in \D^\infty(S^1)$ has been given and that
its rotation number $\alpha=\rho(h)$ is irrational.
Because of Denjoy's theorem, there is a homeomorphism $\varphi$ of $S^1$ conjugating $h$
to the rotation $R_\alpha$, i.e. $\varphi\circ h \circ \varphi^{-1} = R_\alpha$. Since $\alpha$ is
irrational, the centralizer of $R_\alpha$ in $\D^0(S^1)$ is the group of rotations. Hence the centralizer 
$Z^0(h)$ is the set $\{\varphi^{-1}\circ R_\beta \circ \varphi \}$ with $\beta\in S^1$. The
natural inclusion
\begin{equation}\label{rot}
\iota\colon Z^\infty(h) \hookrightarrow Z^0(h) \cong S^1
\end{equation}
is a continuous group morphism mapping  $Z^\infty(h)$ onto a subgroup of $S^1$ 
(cf. \cite{Yoc}, p. 185). 
Note that the map $\iota$ is just the restriction of $\rho$ to $Z^\infty(h)$. 
Notice also that the identification of $Z^\infty(h)$ with a subgroup of $S^1$ given by 
$\iota$ is algebraic but not topological.

Let $d_\infty$ be a distance defining the topology of $\D^\infty(S^1)$ and set 
$$
K_\infty(f) = \inf_{n\geq 1} d_\infty (f^n, \id).
$$

\begin{lemma}
$K_\infty^{-1}(0) = A \cup B$ where $A$ is the subset of  $\D^\infty(S^1)$
of elements $f$ such that $Z_0^\infty(f)$ has the cardinality of the continuum
and $B$ is the subset of  $\D^\infty(S^1)$ of elements of finite order.
\end{lemma}

\begin{proof}
Let $f$ be an element of $K_\infty^{-1}(0) \backslash B$.
Each neighborhood of the identity contains non trivial elements of  $\,\langle f \rangle\,$ and 
therefore of its closure $Z_0^\infty(f)$ with respect to the $C^\infty$-topology. Hence the group 
$Z_0^\infty(f)$ is closed and non-discrete. As $Z_0^\infty(f)$ is perfect, it contains a Cantor set and 
its cardinality is bigger or equal to the cardinality of the continuum $2^{\aleph_0}$. 
In fact $Z_0^\infty(f)$ is separable, as it is a space of smooth functions on a compact manifold, 
and therefore its cardinality is that 
of the continuum.
This shows that $K_\infty^{-1}(0) \subset B \cup A$.

Clearly $B$ is a subset of $K_\infty^{-1}(0)$. Let us show that each element $f\in A$ belongs to 
$K_\infty^{-1}(0)$. The group $Z_0^\infty(f)$ cannot be discrete. Otherwise it would be countable, 
since it is separable, getting a contradiction. We deduce that there is a sequence $(f^{n_k})_{k\geq 1}$
of pairwise different elements that converges. Notice that $K_\infty^{-1}(0)$ is closed under inversion.
Hence we can assume that $n_k >0$. Moreover we can also assume that the sequence 
$n_k$ is strictly increasing. We choose a sequence $m_k= n_{q(k)} - n_k$ such that $q(k) > k$ 
for each $k\in \mathbb N$ and such that $(m_k)_{k\geq 1}$ is strictly increasing. Then 
$(f^{m_k})_{k\geq 1}$ is a sequence of non-trivial elements converging to the identity and 
therefore $f$ belongs to $K_\infty^{-1}(0)$.
\end{proof}

Given an irrational number $\alpha\in S^1-\mathbb Q/\mathbb Z$, set
$$
Z^\infty_0 (\alpha) = \bigcap_{\rho(f) = \alpha} Z^\infty_0(f)
$$
where the groups $Z^\infty_0(f)$ are thought as subsets of $S^1$.
The following result was proved by Yoccoz (cf. \cite{Yoc}, Th\'{e}or\`{e}me 3.5, p. 190)
\begin{theorem}[Yoccoz]
There is a residual set $\mathcal C$ of numbers $\alpha\in S^1-\mathbb Q/\mathbb Z$ such that  
$Z^\infty_0 (\alpha)$ has the cardinality of the continuum. 
\end{theorem}

\begin{corollary}
There is an element  $h_0\in \D^\infty(S^1)$ whose rotation number $\alpha$ is Liouville 
and such that  $Z^\infty_0$ has the cardinality of the continuum. 
Moreover,
each continuous path $\gamma\colon [0,1]\rightarrow Z^\infty(h_0)$ is constant. 
\end{corollary}

\begin{remark}
We recall that for each Liouville number $\alpha$ there is a non-$C^\infty$-linearizable
diffeomorphism of $S^1$ whose rotation number is $\alpha$. This result was proved by Herman in 
\cite[Chapter XI]{Her}. Yoccoz proved that, on the contrary, all diffeomorphisms with a diophantine 
rotation number are $C^\infty$-linearizable \cite{Yoc0}.
  
\end{remark}

\begin{proof}
Since the set of Liouville numbers is residual 
and the intersection of two residual sets
is residual, we can find an element $h_0\in \D^\infty(S^1)$ whose rotation number 
$\alpha = \rho(h_0)$ is Liouville and belongs to the set $\mathcal C$ in the above theorem, 
and such that $h_0$ is not linearizable. Since $\alpha\in \mathcal C$ we deduce that there are 
non trivial elements of $Z^\infty(h_0)$ arbitrarily close to the identity in the $C^\infty$-topology.
Let $\gamma\colon [0,1]\rightarrow Z^\infty(h_0)$ be a continuous path. Since the $C^\infty$-topology
is finer that the $C^0$-topology, it is sufficient to prove that if $\gamma$ is continuous 
when we consider on $Z^\infty(h_0)$ the topology of the uniform convergence, then the path $\gamma$
is constant. But this is equivalent to say that the composition 
$$
\rho\circ\gamma\colon [0,1]\rightarrow Z^\infty(h_0) \hookrightarrow Z^0(h_0) \cong S^1
$$
is constant. If this were not the case the image $(\rho\circ\gamma) [0,1]$ would contain diophantine 
rotation numbers. In particular there would be an element $f$ in $Z^\infty(h_0)$ which is 
$C^\infty$-conjugated to a rotation, but then the same conjugation would also linearize $h_0$
leading to a contradiction. 
\end{proof}

\begin{proof}[Proof of Theorem \ref{contra}]
Let $h_0$ be the element of $\D^\infty(S^1)$ given by the above corollary and let $\F$
be the foliation on $T^2$ obtained by suspension of $h_0$. Assume that $\D_{\F}$
is endowed with a structure of locally path-connected topological group such that the natural inclusion 
$\D_\F \hookrightarrow \D(T^2)$ is continuous. Let $U$ be a neighborhood of 
the identity in $\D(T^2)$ such that the map $f\to P(f)$, given in \eqref{P}, is defined for 
$f\in U\cap \D_\F$. Because of Lemma~\ref{restriction}, $P(U\cap \D_\F)$
contains a neighborhood of the identity in $Z^\infty(h_0)$. This implies that 
$U\cap \D_\F$ has an uncountable number of path-connected components.
Therefore $\D_\F$ cannot be second countable.
\end{proof}

\begin{remark}
Theorem \ref{contra} provides an example of a closed subgroup of $\D(T^2)$ for which the LPSAC-topology 
and the induced topology do not coincide. More precisely, the induced topology of the subgroup is not LPSAC 
and its LPSAC-topology is not countable. This shows in particular that condition b) in Proposition~\ref{two-top}
is not always fulfilled.
%
%
\end{remark}


%

\section{The automorphism group of a transversely holomorphic foliation}
\label{TH}
Throughout this section we assume that the foliation  
$\F$ on the compact manifold $M$ is transversely holomorphic, 
of (real) dimension $p$ and complex codimension $q$.
This means that the atlas of adapted local charts $\{(U_i, \varphi_i)\}$
can be chosen as taking values in $\mathbb R^p\times\C^q$ and that the 
maps $\{\gamma_{ij}\}$ fulfilling the cocycle condition (\ref{cocycle}) are
local holomorphic transformations of $\C^q$.
 
The transverse complex structure of $\F$ induces a complex structure on the normal 
bundle $\nu\F= TM/T\F$ of the foliation, which is invariant by the holonomy.



We denote by $\AF$ the group of automorphisms of the foliation. That is, $\AF$
is the group of elements $f\in\DF$ which are transversely holomorphic in the sense
that, 
in adapted local
coordinates 
they are of the form
$$
f(x,z) = (f_1(x,z,\bar z), f_2(z))
$$
with $f_2 = f_2(z)$ holomorphic transformations.


Let $\aF$ denote the Lie algebra of (real) vector fields $\xi$ on $M$ whose flows are one-parameter subgroups
of $\AF$. In particular, $\aF$ is 
the Lie subalgebra of $\X_\F$. 
%
%

In adapted local coordinates $(x,z)$, a (real) vector field $\xi\in \X$ is written in the form 
\begin{equation}\label{fol_vec}
\xi = \sum_i a^i(x,z,\bar z) \frac{\partial}{\partial x^i} + \sum_j b^j(x, z, \bar z) \frac{\partial}{\partial z^j}
+ \sum_j \overline b^j(x, z, \bar z) \frac{\partial}{\partial \bar z^j},
\end{equation}
where the functions $a^i$ are real whereas $b^j$ are complex-valued.
Then $\xi$ is a foliated vector field, i.e. it belongs to $\X_\F$, if $b^i$ are basic functions, that is if $b^i=b^i(z, \bar z)$.
And $\xi$ is a holomorphic foliated vector field, i.e. it belongs to $\aF$, if $b^i$ are 
holomorphic basic functions, that is if $b^i=b^i(z)$.

Notice that $\D_{L}$ is a normal subgroup of $\AF$ and that $\X_{L}$ is an ideal of $\aF$. 
We denote by $\G = \aF/\X_{L}$ the quotient Lie algebra.
Notice also that   $\D_{L}$ is not necessarily closed in $\AF$
as far as the leaves of $\F$ are not closed.

We will suppose that a Riemannian metric $g$ on $M$, that will play an auxiliary role, has also been fixed. 
\medskip


We recall now some facts concerning
transversely holomorphic foliations that we will use in the sequel 
(for the omitted details we refer to \cite{DK}).
The transverse complex structure of $\F$ induces an almost complex structure on the normal
bundle $\nu\F$ and therefore a splitting of the complexified normal bundle in the standard way
$$
\nu^{\mathbb C} \F = \nu^{1,0} \otimes \nu^{0,1}.
$$
There is a short exact sequence of complex vector bundles
$$
0 \longrightarrow L \longrightarrow T^{\mathbb C} M \overset{\pi^{1,0}} \longrightarrow \nu^{1,0}
\longrightarrow 0
$$
where $\pi^{1,0}$ denotes the composition of natural projections 
$$
T^{\mathbb C}M \to \nu^{\mathbb C} \F \to \nu^{1,0}
$$ 
and $L$ is the kernel of $\pi^{1,0}$. In local adapted
coordinates $(x, z)$, the bundle $\nu^{1,0}$ is spanned by the (classes of) vector fields 
$[\partial/ \partial z^j]$. 
In a similar way, the bundle $L$ is spanned by the vector fields 
$\partial/ \partial x^i, \partial/\partial \bar z^j$, and the dual bundle $L^\ast$ is spanned by the 
(classes of) 1-forms $[dx^i, d\bar z^j]$. 
For a given $k\in \mathbb N$, denote 
$$
\A^k = \Gamma\Big(\bigwedge^k L^\ast \otimes \nu^{1,0}\Big).
$$
%
%
The exterior derivative $d$ induces 
differential operators of order $1$ (cf. \cite{DK})
$$
d^k_\F \colon \A^{k} \longrightarrow \A^{k+1}.
$$
In adapted local coordinates, the operators $d^k_\F$ act as follows

\begin{equation*}
\begin{split}
d^k_\F \Big[\sum\alpha^{\ell}_{IJ} dx^I\wedge d\bar z^J&\otimes \dfrac{\partial}{\partial z^\ell}\Big] 
= 
\sum \Big[d(\alpha^{\ell}_{IJ} dx^I\wedge d\bar z^J)\Big] \otimes \Big[ \dfrac{\partial}{\partial z^\ell}  \Big]  \\[2mm]
&= \sum \Big[ \big(\dfrac{\partial \alpha^{\ell}_{IJ}}{\partial x^i} dx^i + 
\dfrac{\partial \alpha^{\ell}_{IJ}}{\partial \bar z^j}d\bar z^j \big) \wedge
dx^I\wedge d\bar z^J \otimes \dfrac{\partial}{\partial z^\ell}\Big].
\end{split}
\end{equation*}

Notice that $\A^0$ is just the space of smooth sections of the vector bundle $\nu^{1,0}$ and that
$\G$ is the kernel of the map $d^0_\F \colon \A^{0} \longrightarrow \A^{1}$.
One has

\begin{proposition}[\cite{DK}]\label{ell}
The sequence 
\begin{equation}
0 \longrightarrow \A^0  \overset{d^0_\F}\longrightarrow \A^1  \overset{d^1_\F}\longrightarrow \A^2 
\overset{d^2_\F}\longrightarrow \ldots \longrightarrow \A^{p+q} \longrightarrow 0
\end{equation}
is an elliptic complex.
\end{proposition}


The elements of the quotient Lie algebra $\G = \aF/\X_{L}$
are naturally identified to the holomorphic basic sections of the normal bundle
$\nu^{1,0}$, and are called holomorphic basic vector fields. In particular $\G$ is 
a complex Lie algebra in a natural way.
Using the Riemannian metric on $M$, we can identify $\G$ to the vector space 
$\X_N$ of those vector fields in $\X$ that are orthogonal to $\F$.
Notice that, although  $\X_N$ is not a Lie subalgebra of $\X$, 
there is a vector space decomposition $\X =\X_{L} \oplus \X_N$.

As a corollary of Proposition~\ref{ell} one obtains the following result, which was also
proved by a different method by X.~G\'{o}mez-Mont \cite{GM}.

\begin{proposition}[G\'{o}mez-Mont, and Duchamp and Kalka]
Let $\F$ be a transversely holomorphic foliation on a compact manifold $M$. 
The Lie algebra $\G$ of holomorphic basic vector fields has finite dimension.
\end{proposition}

We denote by $G$ the simply-connected complex Lie group whose Lie algebra is $\G$.
%


Let $\D_{L,0}$ denote the connected component of $\D_{L}$ containing the identity. 
Our main result concerning the structure of $\AF$ is the following

\begin{theorem}\label{LieGroup_1}
Let $\F$ be a transversely holomorphic foliation on a compact manifold $M$. 
Endowed with the topology induced by $\D$, the automorphism group $\AF$ of 
automorphisms of $\F$ is
a closed, strong ILH-Lie subgroup of $\D$ with Lie algebra $\aF$. 
Moreover, the left cosets of the subgroup $\D_{L,0}$  
define a Lie foliation $\F_{\mathcal D}$ on $\AF$, which is transversely modeled on 
the simply-connected complex Lie group $G$
associated to the Lie algebra $\G$.
\end{theorem}


We recall that a Lie foliation, modeled on a Lie group $G$, is defined 
by local submersions $\Phi_i$ with values on $G$ fulfilling $\Phi_j = L_{\gamma_{ji}}\circ \Phi_i$, 
where $\gamma_{ij}$ is a locally constant function with values in  $G$
and $L_{\gamma_{ij}}$ denotes left translation by $\gamma_{ij}$.
\medskip

The proof of the above theorem will be given in several steps:
\begin{enumerate}[1)]
\item First, we use Theorem~\ref{Frobenius}
to show that there is a strong ILH-Lie subgroup 
$\operatorname{Aut}'(\F)$ of $\D$, contained in $\operatorname{Aut}(\F)$ and whose Lie algebra is $\aF$.
Note that this is not sufficient to assure that $\operatorname{Aut}'(\F)$ coincides with the 
connected component of the identity
of $\operatorname{Aut}(\F)$.
\item We consider the group $\operatorname{Aut}(\F)$ endowed with the topology induced by $\D$ and the
group $\D_{L}$ with the Fr\'{e}chet topology given by Proposition~\ref{d-tangent}, and we fix small enough 
neighborhoods of the identity
$\mathcal V$ and $\mathcal V_L$ of these two groups. We construct a continuous map
$\Phi\colon \mathcal V\to \D$ with image $\tilde\Sigma = \Phi(\mathcal V)$ contained in  $\operatorname{Aut}(\F)$
and with the property that $\mathcal V$ is homeomorphic to the product 
$\mathcal V_{L}\times \tilde\Sigma$.
\item We show that $\tilde \Sigma$ is naturally identified to a neigbourhood $\Sigma$ of $e$ in the Lie group $G$.
This implies that the induced topology of $\operatorname{Aut}(\F)$ is LPSAC and reasoning as in the proof of 
Proposition~\ref{d-tangent} we conclude that, with the induced topology, $\operatorname{Aut}(\F)$ is 
a closed strong ILH-Lie subgroup of $\D$, proving the first part of the Theorem.
\item Finally, we prove that the map $\Phi$ is smooth with respect to the Fr\'{e}chet structure of $\operatorname{Aut}(\F)$
and that the identification of $\mathcal V$ with $\mathcal V_{L}\times \Sigma$ is a local chart of 
$\operatorname{Aut}(\F)$ defining the Lie foliation in a neighborhood of the identity. A similar construction of 
local charts around any element of $\operatorname{Aut}(\F)$ completes the proof.
\end{enumerate}

%
%
%

Thus we begin with showing the following statement.

\begin{proposition}\label{P0}
There is a strong ILH-Lie subgroup $\operatorname{Aut}'(\F)$ of $\D$ whose Lie algebra is $\aF$.
It is contained in $\AF$.
\end{proposition}

\begin{proof}
Let $\pi\colon \X =\Gamma(TM) \to \A^0 = \Gamma(\nu^{0,1})$ be the natural projection. That is, 
if $\xi\in\X$ is written in local adapted coordinates as in \eqref{fol_vec}, then 
$$
\pi (\xi) = \sum_j b^j(x, z, \bar z) \frac{\partial}{\partial z^j}.
$$ 
Then $\pi$ is a differential operator of order $0$ and its adjoint operator $\pi^\ast$,
constructed using suitable Riemannian metrics on $TM$ and $\nu^{0,1}$, is also a differential 
operator of order $0$
such that $\pi\pi^\ast =\id$.

We define the operator $A\colon \Gamma(TM) \to \A^1$ as the composition $A= d^0_\F\circ \pi$ and
we set $B= d^1_\F$. Then one has $B A=0$ and $\ker A = \aF$. Moreover the operator
$$
A\, A^\ast + B^\ast\, B = d^0_\F \circ \delta^0 + \delta^1\circ d^1_\F,
$$
where $\delta^k$ are formal adjoints of $d^k_\F$, is elliptic because of Proposition~\ref{ell}.
Then Theorem~\ref{Frobenius} proves the existence of a strong ILH-Lie subgroup 
$\operatorname{Aut}'(\F)$ of $\D$ with Lie algebra is $\aF$. The fact that 
$\operatorname{Aut}'(\F)$ is an integral submanifold of the distribution $T\D$ obtained
by right translation of $\aF$ implies that each smooth curve $c=c(t)$ in $\operatorname{Aut}'(\F)$
starting at the identity fulfills an equation of type
$$
\dot c(t) = \xi_t \circ c(t)
$$
where $\xi_t$ is a curve in $\aF$. This implies that each diffeomorphism $c(t)$ preserves the 
transversely holomorphic foliation $\F$ showing that $\operatorname{Aut}'(\F)$
is contained in $\AF$.
\end{proof}

\begin{remark}
The above argument also shows that the 
Sobolev completion $\operatorname{Aut}'^k(\F)$ of $\operatorname{Aut}'(\F)$
is also a subgroup of $\operatorname{Aut}^k(\F)$. 
%
\end{remark}


\begin{proposition}\label{P1}
Let us consider $\AF$ endowed with the topology induced by $\D$, and $\D_{L}$ endowed with the 
Fr\'{e}chet topology given by Proposition~\ref{d-tangent}.
There are open neighborhoods of the identity $\mathcal V\subset \AF$ and 
$\mathcal V_L \subset \D_{L}$, and a continuous map 
$\Phi\colon\mathcal V\to  \D$, with image  $\tilde\Sigma = \Phi(\mathcal V)$ contained in $\AF$,
such that the multiplication map 
\begin{equation}\label{morph}
\begin{array}{rccc}
\Psi\colon& \mathcal V_L\times \tilde\Sigma &\longrightarrow & \mathcal V \\
 & (f, s) & \mapsto & f\circ s
\end{array}  
\end{equation}
is a homeomorphism
and the following diagram commutes
\begin{equation}\label{cdiag}
\xymatrix{
\mathcal V_L\times \tilde\Sigma \ar[r]^{\Psi} \ar[d]_{pr_2} & \mathcal V \ar[dl]^{\Phi} \\ 
\tilde\Sigma & 
}
\end{equation}
\end{proposition}


\begin{proof}
The idea for proving this proposition is to associate, to any given diffeomorphism  
$f\in\AF$ close enough to the identity, 
the element
$\bar f\in \AF$ fulfilling the following property: for each $w$, 
$\bar f(w)$ is the point of the local slice of the foliation through $f(w)$ which is 
at the minimal distance from $w$.
%
%
The mapping $\bar f$ is uniquely determined by that condition and 
the diffeomorphisms $f$ and $\bar f$ 
differ by an element of $\D_{L}$ close to the identity. Then $\tilde\Sigma$ will be the set of the diffeomorphisms
$\bar f$ obtained in this way. In order to carry out this construction in a rigorous way we need to introduce
appropriate families of adapted local charts.

We say that a local chart $(U,\varphi)$ of $M$ adapted to $\F$ is cubic 
if its image  
$\varphi(U)\subset \mathbb R^p\times \C^q$ is the product $C\times\Delta$
of a cube $C\subset\mathbb R^p$ and a polydisc $\Delta\subset\C^q$, and we say that it is 
centered at $w_0\in U$ if $\varphi(w_0) = (0,0)$.  The submersion $\phi=pr_2\circ \varphi$ 
maps $U$ onto $\Delta$ and the slices $S_z = \phi^{-1}(z)$ in $U$ are copies of $C$.

Let $N\F\cong \nu\F$ be the subbundle of $TM$ orthogonal to $T\F$ 
and let $p\colon N\F\rightarrow M$ be the natural projection.
We denote by $\exp\colon N\F\rightarrow M$ the geodesic exponential map
associated to the Riemannian metric $g$ and we set 
$W_\epsilon= \{(w,v)\in N\F\mid |v| <\epsilon\}$. Since $M$ is compact there is $\epsilon_1>0$
such that $T_w=\exp(p^{-1}(w)\cap W_{\epsilon_1})$ is a submanifold transverse to $\F$ for each $w\in M$.

Given a cubic adapted local chart 
$(U,\varphi)$, we set $W_{\epsilon,U}=W_\epsilon\cap p^{-1}(U)$ and 
we define $\psi_U\colon W_{\epsilon,U}\rightarrow \Delta\times M$ as the map 
\begin{equation}
\psi_U(w,v) = (\phi(w), \exp(w,v)).
\end{equation}
Then there are constants $0<\epsilon_0<\epsilon_1$ and $0<\delta_0$ such that,
if the diameter of $U$ fulfills $\operatorname{diam} U<\delta_0$, then $\psi_U$ is a 
diffeomorphism from $W_{\epsilon_0,U}$ onto its image.
We fix once and for all the values $\epsilon_0, \delta_0$ 
and we set $W_U=W_{\epsilon_0,U}$. We denote by $W_U^z$ the restriction of 
$W_U$ to the slice $S_z$, i.e. $W_U^z = \{(w,v)\in W_U \mid \phi(w) = z\}$.

Let $(U,\varphi)$ be a cubic adapted chart centered at $w_0$ and $U'$ an open subset 
of $U$. 
We say that the pair $(U, U')$ is {\emph{regular}} if the following conditions are fulfilled:
\smallskip

(R1) $\operatorname{diam} U<\delta_0$,

(R2) $w_0\in U'\subset \overline U'\subset U$ and $(U', \varphi|_{U'})$ is a cubic chart,

(R3) $\overline U'$ is contained in the image $\exp(W_U^z)$ for each $z\in \Delta$.
\smallskip

\noindent If the pair $(U, U')$ is regular then the intersection of a slice $S_z$ of $U$ 
with $U'$ is either empty or coincides with a slice of $U'$. Condition (R3) implies that, 
for each point $w\in U'$ and each slice $S_z$ of $U$ meeting $U'$, there is a unique point $w'\in S_z$ 
which is at the minimal distance between $w$ and $S_z$, namely
\begin{equation}\label{minimal}
w' = p\circ \psi_U^{-1}(z,w).
\end{equation}
Notice also that the submersion $\phi$
maps the intersection $T_w^{U'}= T_w\cap U'$ diffeomorphically onto the open 
polydisc $\phi(U')\subset\C^q$.

Since $M$ is compact, we can find positive numbers $0<\delta'<\delta <\delta_0$,  
%
a finite family of adapted and cubic local charts $\{(U_i, \varphi_i)\}_{ i=1,\dots, m}$, centered at points $w_i$,
and a family of 4-tuples 
\begin{equation}
\mathcal U = \{(U_i, U'_i, V_i, V'_i)\}_{ i=1,\dots, m}
\end{equation}
of open sets such that
\smallskip

(C1) $(U_i, U'_i)$ is a regular pair fulfilling $\operatorname{diam} U_i<\delta$ and $\delta'< d(w_i, \partial U'_i)$,

(C2) $\overline V_i\subset U'_i$, $(V_i, \varphi_i|_{V_i})$ is a cubic adapted local chart centered at $w_i$, 
$\operatorname{diam} V_i<\delta'/4$ and $(V_i, V'_i)$ is a regular pair,

(C3) the family $\{ V'_i\}_{ i=1,\dots, m}$ is an open covering of $M$.
\smallskip

\noindent We then set $\psi_i =\psi_{U_i}$, $W_i= W_{U_i}$ and $W_i^z= W_{U_i}^z$. 
We remark that such a covering has the following properties:
\smallskip

(P1) If $V_i\cap V_j \neq \emptyset$ then $\overline{V}_i \cup \overline{V}_j \subset U'_i\cap U'_j$,

(P2) given a point $w\in V'_i$ and a slice $S_z$ in $U_i$ which meets $V_i$,
the point $w'\in S_z$ that is at the minimal distance from $w$ belongs to $V_i$.
\medskip

We fix from now on such a family $\mathcal U$.
%
We denote by
$\mathcal V$ the open neighborhood of the identity in $\AF$ defined by
$f\in \mathcal V$ if and only if $f(\overline V'_i)\subset V_i$ for each $i=1,\dots m$.
We also define the subset $\mathcal V_L$ of $\D_{L}$ of those elements $f\in \D_{L}$
fulfilling $f(\overline V'_i)\subset V_i$ for each $i=1,\dots m$ and keeping fixed each 
slice of $V'_i$ (recall that a leaf of $\F$ can cut $V_i$ in many slices). The set 
$\mathcal V_L$ is open in the Fr\'{e}chet topology of $\D_{L}$. 
We are now ready to define the map $\Phi\colon \mathcal V \to \AF\subset\D$
and the set $\tilde\Sigma = \Phi(\mathcal V)$.
%
%
%
%

Let $f$ be a fixed element in $\mathcal V$. For a given point $w\in M$ we choose
$i=1,\dots ,k$ with $w\in V'_i$. Then $\hat w = f(w)\in V_i$ and we denote by $S_z$ and $S_{\hat z}$
the slices of $U_i$ through the points $w$ and $\hat w$ respectively, i.e. $z=\phi_i(w)$ 
and $\hat z=\phi_i(f(w))$. 
Then $S_{\hat z}$ meets  $V_i\subset U'_i$ and, since the pair $(U_i, U'_i)$ is regular,
there is a uniquely determined point 
$w'\in S_{\hat z}$ which is at the minimal distance from  $S_{\hat z}$ to $w$, 
moreover $w'\in V_i$. We then define
a local map $\bar f\colon V'_i\to V_i\subset M$ by $\bar f(w) = w'$. 
It follows from properties (P1) and (P2) that the above definition does not depend on the 
choice of  $i=1, \dots, m$. Indeed, if $w\in V'_i\cap V'_j$ then the points $w'$ constructed using $U_i$
or $U_j$ belong to $V_i \cap V_j$ and are necessarily the same. 
This fact implies that the map $\bar f$ is globally defined in 
the entire manifold $M$. 

Since $f$ preserves leaves and is transversely holomorphic, 
it induces a local holomorphic transformation $\hat f_i$  between open subsets of 
$\Delta_i = \phi_i(U_i) \subset \C^n$. 
Notice that, by construction, $\bar f$ sends leaves into leaves and that 
the local transverse part of $\bar f$ in $U_i$ is just $\hat f_i$.

Let us prove that $\bar f$ is a smooth diffeomorphism. We first notice that, if $w'=\bar f(w)$ with 
$w\in V'_i$, then the geodesic joining $w$ and $w'$ and giving the minimal distance
between $w$ and the slice $S_{z'}$, where $z'= \phi_i(w')$, meets $S_{z'}$
orthogonally. By applying \eqref{minimal}, we have
\begin{equation}\label{intermedi}
\bar f(w) = w' = p\circ\psi_i^{-1}(\phi_i(f(w)), w), 
\end{equation}
Using the local transverse part $\hat f_i$ of $f$ the above equality can be written as
%
%
%
%
%
\begin{equation}\label{themap}
\bar f(w) = w' = p\big(\psi_i^{-1}(\hat f_i(\phi_i(w)), w)\big).
\end{equation}
This equality proves that $\bar f$ is smooth and that $\bar f$ is entirely determined by the set 
of local transformations $\{\hat f_i\}$.

Restricted to $V'_i$ the map $\bar f$
has an inverse that can be described as follows. Given $w'=\bar f(w)\in V_i$
we put $z'= \phi_i(w')$ and $z= \phi_i(w)=\hat f_i^{-1} (z')$. As pointed out before, 
$\exp(N_{w'}\F\cap W_i)$ meets the slice $S_z$ and the intersection
is precisely the point $w\in S_z$. In fact, there is a 
unique vector
$v=v(w')\in N_{w'}\F\cap W_i$ such that  $w = \exp(w',v)$. Therefore we can write
$$
w = \bar f^{-1}(w') = \exp(w',v) = pr_2\circ \psi_i (w', v(w')) 
$$
and, since $\psi_i$ is a diffeomorphism, the vector
$v=v(w')$ depends smoothly on $w'$. Indeed, if we set $T= \exp(N_{w'}\F\cap W_i)\cap V_i$
then $\phi_i|_{T}$ maps $T$ isomorphically onto $\phi_i(V_i)\subset \Delta_i$ and we have
$$
(w', v(w')) = \psi_i^{-1}\big(z, (\phi_i|_{T})^{-1}(z)\big).
$$
We deduce in particular that $\bar f$ is locally injective 
and that its (local) inverse is smooth. This also proves that the map 
$\bar f$ is open and, as the manifold $M$ is compact, $\bar f$ is necessarily
surjective.

We remark that formula (\ref{themap}) also shows that $\bar f$ is entirely determined by the set 
of local transformations $\{\hat f_i\}$.
That formula also shows that
$\bar f$ depends continuously on $f\in\mathcal U$.
Since the transverse part of $\bar f$ is just $\hat f$, the map $\bar f$
is transversely holomorphic.  It remains to prove that $\bar f$ is globally 
injective and therefore an element of $\AF\subset \D$. 

Assume that $w' = \bar f(w_1) = \bar f(w_2)$ and choose $i, j \in \{1, \dots, m\}$
with $w_1\in V'_i$ and $w_2\in V'_j$. Then $V_i\cap V_j\neq \emptyset$ and therefore
$V_i\cap V_j\subset U'_i$. We deduce that $w_1, w_2$ belong to the same slice
$S_z$ of $U_i$ where $z = \hat f_i^{-1} (\phi_i(w'))$. Now $w_1, w_2$ are in the intersection 
of $S_z$ with the submanifold $\exp(N_{w'}\F\cap W_i)\cap V_i$. 
But this intersection reduces to a unique point, which shows that $w_1=w_2$.

Summarizing the above considerations, we see that the correspondence 
$f\mapsto \bar f  = \Phi(f)$
given by \eqref{themap} defines a continuous map $\Phi\colon \mathcal V \to \D$
with image $\tilde \Sigma = \Phi(\mathcal V)$ contained in $\AF$.
By construction, 
the diffeomorphism $f_L:= f\circ\bar f^{-1}$ is an element of $\mathcal V_L$
and $\Phi$ maps $\mathcal V_L$ to the identity. 
%
%
%
%
It is clear from the construction that each element
$f\in\mathcal V$ decomposes in a unique way as the composition 
$$
f = f_L \circ  \bar f
$$
with $\bar f = \Phi(f)\in \tilde\Sigma$ and $f_L\in \mathcal V_L$ and that $\bar f$ and $f_L$
depend continuously on $f$, with respect to the Fr\'{e}chet topology. 
This tells us that the composition map 
$\Psi\colon\mathcal V_L\times \tilde\Sigma \rightarrow  \mathcal V$ defined in (\ref{morph})
has a continuous inverse. Therefore, by shrinking the set $\mathcal V_L$ if it is necessary,
the map $\Phi$ is a homeomorphism and the diagram (\ref{cdiag}) is commutative. 
%
\end{proof}

\begin{remark}
Notice that $\tilde\Sigma$ is not necessarily closed under composition or taking inverses,
even for elements close to the identity.
%
%
\end{remark}

\begin{remark}\label{cdiag_k}
The construction carried out in the proof of the above proposition also provides a commutative 
diagram of continuous maps
\begin{equation}
\xymatrix{
\mathcal V^k_L\times \tilde\Sigma^k \ar[r]^{\Psi} \ar[d]_{pr_2} & \mathcal V^k \ar[dl]^{\Phi} \\ 
\tilde\Sigma^k & 
}
\end{equation}
where $\mathcal V^k$ and $\mathcal V^k_L$ are open neighborhoods of the identity in the Sobolev completions 
$\operatorname{Aut}^k(\F)$ and $\D^k_L$ respectively, the maps $\Phi$ and $\Psi$ are defined in the same way 
and $\Psi$ is a homeomorphism.
\end{remark}

\begin{remark}\label{D_F}
Notice that the correspondence $f\mapsto \bar{f} = \Phi(f)$ carried out in the above theorem is well defined 
for any $f\in \D_{\F}$ close enough to the identity since the transverse holomorphy of the diffeomorphisms
plays no role in the construction of $\Phi$. However, for a general foliation, nothing can be said about the properties
of the image $\tilde\Sigma$ of the map so constructed. 
\end{remark}

The following proposition states that, after shrinking it, if necessary, the set $\tilde\Sigma$ is naturally
identified with an open neighborhood of the unity in the simply-connected Lie group $G$ associated 
to the Lie algebra $\G$. More precisely, if we identify $\G$ with the space $\X_{N,b}$ of 
transversely holomorphic basic vector fields on $M$ that are orthogonal to $\F$, then 
we have:

\begin{proposition}\label{P2}
There is a neighborhood $\Sigma$ of the unity in $G$ and a homeomorphism 
$\sigma\colon\Sigma\to\tilde\Sigma$ 
such that the composition
$$\hat\Phi:=\sigma^{-1}\circ\Phi \colon \mathcal V\to \Sigma$$
preserves multiplication and inverses, i.e. $\hat\Phi$ is a morphism of local 
topological groups. 
Moreover, the composition 
$$U\subset \G \equiv\X_{N,b} \to \Sigma,$$ 
defined as $[\xi]\to\hat\Phi\circ \varphi_1^\xi$
for vector fields $\xi$ in $\X_{N,b}$ close to zero,
is the restriction to a neighborhood of zero of the natural exponential map $\G\to G$.
\end{proposition}

Here, $\varphi_t^\xi$ stands for the one-parameter group associated to $\xi$.

\medskip
\noindent{\it Proof.}
We assume, as in the proof of the above proposition, that a finite
family of adapted and cubic local charts $\{(U_i, \varphi_i)\}_{ i=1,\dots, m}$, and 
a family of 4-tuples $\{(U_i, U'_i, V_i, V'_i)\}_{ i=1,\dots, m}$ fulfilling conditions C1, C2 and C3, 
have been fixed. As before we set $\phi= pr_2\circ\varphi$ and we denote by $\Delta_i$ the 
polydisc $\phi_i(U_i)$ of $\mathbb C^q$.
We also set $D_i = \phi_i(V_i)$ and $D'_i = \phi_i(V'_i)$. Let $B_i$ be the Banach space of 
continuous maps $f\colon \bar D_i \to \mathbb C^q$ that are holomorphic on $D_i$ with the norm
$$
| f |_i = \max _{z\in \bar D_i} |f(z)|.
$$
Then the space $\mathcal B = \oplus_i B_i$ with the norm
$$
\| F \| = \| (f_1, \dots, f_m)\| = \sum_i |f_i|
$$
is also a Banach space. Notice that $I = (\id_1, \dots, \id_m)$, where $\id_i$ denotes the identity map of $D_i$,
is an element of $\mathcal B$. 

We claim that there is an open neighborhood $\mathcal I_\epsilon = \{F\in \mathcal B\mid \|F-I\|<\epsilon\}$
of $I$ in $\mathcal B$ such that each $F= (f_1, \dots,f_m)\in\mathcal I_\epsilon$ fulfills: (i) $f_i(\bar D'_i) \subset D_i$
and (ii) $f_i|_{D'_i}$ is a biholomorphism onto its image for $i=1, \dots ,m$. Indeed, we know from 
Cauchy's inequalities that, for a given $\epsilon'>0$ there is $\epsilon>0$ such that $\|F-I\|<\epsilon$ implies
that $\|f_i-\id_i\|_{C^1}<\epsilon'$, where $\|\cdot\|_{C^1}$ stands for the $C^1$-norm on the polydisc $D'_i$. 
%
%
%
%
%
Using the mean value theorem we see that, for $z, z'\in D'_i$,
one has 
\begin{equation}\label{estimate1}
|f_i(z) - f_i(z') | \geq c \, |z-z'|
\end{equation}
for a suitable constant $c>0$.
Inequality (\ref{estimate1})
implies that the restriction of $f_i$ to the polydisc $D'_i$ is injective and it is known 
that a one-to-one holomorphic map 
is necessarily a biholomorphism onto its image. Therefore, there is $\epsilon<0$ such that condition (ii) holds for 
$F\in \mathcal I_\epsilon$ and condition (i) is also fulfilled if $\epsilon$ is small enough. We assume that such an 
$\epsilon$, and therefore the open subset $\mathcal I_\epsilon$ of $\mathcal B$,
have been fixed.

We consider now the holonomy pseudogroup $\mathcal H$ of the transversely holomorphic foliation $\F$.
As a total transversal of $\F$ we choose the union $T = \cup T_i$ of disjoint submanifolds  $T_i$ of $X$ of dimension $2q$
and transverse to $\F$ with the property that $T_i \subset V'_i$ and $\phi_i(T_i) = D'_i$. Notice that $T$ is naturally 
endowed with a complex structure. We identify $T$, through the maps $\phi_i$, with the disjoint union of open polydiscs 
$T\equiv\amalg\, D'_i$. The holonomy of the foliation $\F$ induces local transformations between open subsets of the total transversal $T$.
They are holomorphic and $\mathcal H$ is the pseudogroup generated by these holonomy transformations. Since the manifold 
$X$ is compact, $\mathcal H$ is a compactly generated pseudogroup. This means that $\mathcal H$ is generated by a 
finite number of transformations of the pseudogroup, $h_\mu\colon W_\mu \to T$ for $\mu = 1, \dots, \ell$, 
with the following properties: $W_\mu$ 
is a relatively compact subset of some $D'_i$ and 
$h_\mu$ has an extension to an open neigborhood of $\overline{W}_\mu$ in $D'_i$ 
such that $h_\mu (\overline{W}_\mu) \subset D'_j$ for some
$1 \leq j \leq m$ (cf. \cite{Hae}).

We define $\mathcal S$ as the subset of $\mathcal I_\epsilon$ of those $F=(f_1, \dots, f_m)$
fulfilling
\begin{equation}\label{global}
f_j^{-1}\circ h_\mu\circ f_i = h_\mu
\end{equation}
in the common domain of definition. 
Clearly, $\mathcal S$ is closed in $\mathcal I_\epsilon$. We notice that each element 
$F=(f_1, \dots, f_m)\in \mathcal S$
induces a well-defined diffeomorphism $\sigma(F)$ belonging to $\tilde\Sigma$.
This correspondence $\sigma\colon \mathcal S \to \tilde\Sigma$ 
can be constructed as follows. Formula (\ref{themap}), where we replace $\hat f_i$ 
by the corresponding component $f_i$ of $F$, determines a (local) foliation preserving 
diffeomorphism $\bar f_i\colon V'_i\to V_i$ whose transverse part is just $f_i$. The identities
(\ref{global}) guarantee that formula \eqref{themap} applied to the family $\{f_i\}$ provides a well defined 
diffeomorphism 
$f= \sigma(F)$ of $M$ belonging to $\tilde\Sigma$. Clearly $\sigma$ maps
$\mathcal S$ homeomorphically  onto its image 
$\tilde\Sigma' = \sigma(\mathcal S)$.
The fact that $f= \sigma(F)$ is a diffeomorphism globally defined over $M$
implies in 
particular that each one of the maps $f_i$ extends holomorphically to the polydisc 
$\Delta_i = \phi_i(U_i)$ and maps $\phi_i(U'_i)$ biholomorphically 
onto an open subset of $\Delta_i$. From that fact one deduces easily that 
$\mathcal S$ is a local topological group (for the general properties of local 
topological groups we refer to \cite{Pont}).
Moreover, if we set $\mathcal V' = \Phi^{-1}(\tilde\Sigma')$
then the composition 
$$\hat\Phi:=\sigma^{-1}\circ\Phi \colon \mathcal V'\to \mathcal S$$
can be seen as the map which associates 
its transverse part $\hat f$ to each diffeomorphism
$f$ in $\AF$ which is close enough to the identity. It follows in particular that the map $\hat\Phi$
preserves multiplication and inverses. It is therefore a morphism of local groups.

The end of the proof will be based on the following fact.

\begin{lemma}\label{montel}
The local group $\mathcal S$ is locally compact. 
\end{lemma}

\begin{proof}
Let us fix $0<\epsilon'<\epsilon$.
It is sufficient to prove 
that the neighborhood $\mathcal S_{\epsilon'}$ of 
$I$ defined by $\|F-I\|\leq\epsilon'$ is compact. 
Let  $\{F_k =(f_{k,1}, \dots, f_{k,m})\}$ 
be a sequence of elements of $\mathcal S_{\epsilon'}$. 
As noticed before, we can think of each $f_{k,i}$
as a holomorphic transformation defined on  $\Delta_i$ that sends 
$\phi_i(U'_i)$ into $\Delta_i$. By 
taking $\epsilon'$ small enough, we can also assume that $f_{k,i}(\phi_i(U'_i))$ contains $D_i$. 
It follows from Montel's theorem that there is a subsequence 
$\{F_{k_j}\}$ of $\{F_k\}$ such that each $f_{{k_j},i}$ converges uniformly 
to a holomorphic 
map $\tilde f_i$. Since $f_{k,i}$ are invertible, we can suppose that the sequences 
$\{f_{k_j}^{-1}\}$, that are defined on $D_i$, also converge to a limit which
necessarily is $\tilde f^{-1}_i$. This proves that $\{F_{k_j}\}$ has a limit that belongs to 
$\mathcal S_{\epsilon'}$.  
\end{proof}

\begin{proof}[End of proof of Proposition \ref{P2}]
Now we are in position to apply Theorem A in \cite{BM1} to the local group 
$\mathcal S$.
Although the result by Bochner and Montgomery
is stated in the context of topological groups, the arguments used there are purely local and therefore 
they still remain valid for locally compact 
local groups of transformations. It implies in particular that $\mathcal S$ is a local Lie 
group, i.e. $\mathcal S$ is isomorphic to a neighborhood of the identity of a certain Lie group. 

Finally, we claim that this Lie group is just the 1-connected Lie group $G$ associated 
to the Lie algebra $\G$. Indeed, each element $\xi$ of $\X_{N,b}\equiv \G$ induces
a one parameter group $\varphi_t^\xi$ that is sent by $\hat\Phi$ (and for small $t$)
into a local one parameter subgroup of $\mathcal S$ which does not reduce to the identity
unless $\xi = 0$. Conversely, each local one parameter subgroup of $\mathcal S$
is obtained in this way. Hence $\mathcal S$ is naturally identified to a 
neighborhood $\Sigma$ of the unity in $G$. The last assertion of the proposition is clear 
from the above discussion.
\end{proof}

\begin{remark}\label{sigma_k}
The map $\sigma$ constructed in the above proposition also identifies $\Sigma$ 
with the set $\tilde\Sigma^k = \Phi(\mathcal V^k)$ considered in Remark~\ref{cdiag_k}
\end{remark}

\begin{remark}\label{diff}
The following criterium is useful to assure that certain maps between Fr\'{e}chet 
manifolds are smooth (cf. \cite{Ham}).
Let $V$ and $V'$ be fibre bundles over a compact manifold $X$, eventually with boundary, 
and let $\Gamma(V)$
and $\Gamma(V')$ denote the Fr\'{e}chet manifolds of the 
smooth sections of $V$ and $V'$ respectively.
Assume that there is an open subset $U\subset V$ meeting all the fibres of $V$ and a 
fibrewise smooth map $F\colon U\to V'$
projecting onto the identity of $X$. Let $\tilde U$ be the subset of $\Gamma(V)$ 
of those sections $s$ having image in $U$.
Then the map $\chi\colon\tilde U\subset \Gamma(V) \to \Gamma(V')$
given by $s\mapsto \chi(s) = F\circ s$ is a smooth map of Fr\'{e}chet manifolds.
Moreover, assume that $\mathcal E$ and $\mathcal E'$ are Fr\'{e}chet submanifolds 
of $\Gamma(V)$ and $\Gamma(V')$ respectively with $\mathcal E\subset \tilde U$
and such that $\chi(\mathcal E) \subset \mathcal E'$, then the restricted map
$\chi\colon \mathcal E \to \mathcal E'$ is also smooth.
\end{remark}

Combining the above two propositions we obtain the commutative diagram
\begin{equation}\label{cdiag2}
\xymatrix{
\mathcal V_L\times \Sigma \ar[r]^{\hat\Psi} \ar[d]_{pr_2} & \mathcal V \ar[dl]^{\hat\Phi} \\ 
\Sigma & 
} 
\end{equation}
As a corollary we deduce the following

\begin{proposition}\label{P3}
With the induced topology, $\AF$ is a closed strong ILH-Lie subgroup of $\D$
with Lie algebra $\aF$.
\end{proposition}

\begin{proof}
Notice first that $\AF$ is a closed subgroup of $\D$.  
Using the above commutative diagram as well as Remarks~\ref{cdiag_k}
and \ref{sigma_k}, and arguing as in the proof of Proposition~\ref{d-tangent},
we see that $\operatorname{Aut}'^k(\F)$ coincides with
$\operatorname{Aut}_0^k(\F)$ (the connected component of 
$\operatorname{Aut}^k(\F)$
containing the identity). This implies that $\operatorname{Aut}'(\F)$
is the connected component of the identity of $\AF$ concluding the proof.
\end{proof}


\begin{proof}[End of proof of Theorem~\ref{LieGroup_1}]
We first prove that the map $\hat\Phi$ in diagram~\eqref{cdiag2} is smooth.
Let $\{(U_i, \varphi_i)\}_{ i=1,\dots, m}$ be a finite
family of adapted and cubic local charts of $M$ and let $\{(U_i, U'_i)\}$ be a family of
regular pairs with the properties that $\{U'_i\}$
is an open cover of $M$ and $\phi_i(U_i) = D_i$ and $\phi_i(U'_i) = D'_i$ are open polydiscs. 
We choose a family $\{T_i\}$ of disjoint transversals
$T_i\subset U_i$ of $\F$ such that $\phi_i(T_i) = D_i$. We set $T'_i = T_i \cap \bar U'_i$ and we 
identify $T= \amalg\, T'_i$, through the maps $\phi_i$,
with the disjoint union of closed polydiscs $T= \amalg\, \bar D'_i$. The manifolds $V=\amalg\, (T'_i\times U_i)$
and $W= \amalg\, (T'_i\times D_i)$ are fibre bundles over $T$ taking as projection the natural projections 
onto the first factor. The map $F\colon V\to W$ given by $F(z_i, w_i) = (z_i, \phi_i(w_i))$ is fibrewise and
smooth, and it follows from the criterium stated in Remark~\ref{diff} that the map 
$\chi\colon\Gamma(V) \to \Gamma(W)$, given by $\chi(s) = F\circ s$,
is smooth. Notice that diffeomorphisms of M close to the identity
can be thought of sections of the fibre bundle $pr_1\colon M\times M \to M$ with image close to the diagonal.
Moreover, for a small enough neighborhood $\mathcal W$ 
of the identity in $\D$, the map 
$$
\begin{array}{ccc}
\tau\colon\mathcal W \ & \longrightarrow & \Gamma(V)\\[1mm]
f & \mapsto & \{z_i \mapsto (z_i, f(z_i))\}
\end{array}
$$
is well defined and smooth. If $f\in \mathcal W$ belongs to $\AF$ the family $\{z_i, \phi_i (f(z_i))\} = \chi(\tau(f))$
is an element of the local group $\mathcal S\equiv \Sigma$ constructed in the proof of Proposition~\ref{P2}. 
Therefore 
the map $\hat\Phi\colon \mathcal V\to \Sigma$ can be written as the 
restriction to $\mathcal V\subset \AF$ of the composition $\chi\circ\tau$, which proves that it is smooth.

Let us consider now the map $\sigma\colon \Sigma\equiv \mathcal S\to \tilde\Sigma\subset\AF$. A reasoning similar to 
the previous one, using now the local expression of the correspondence $s=\{f_i\}\in \mathcal S\mapsto \bar{f} \in \tilde \Sigma$
given by the equality~\eqref{themap}, shows that 
$\sigma$ is smooth and therefore that $\hat \Psi \colon \mathcal V_L\times \Sigma \to \mathcal V$, which is given by 
the multiplication $\hat\Psi(f,s) = f\circ\sigma(s)$, is also smooth.
Its inverse $\hat\Psi^{-1}$ is just the correspondence $f\mapsto (f_L, f_N)$,
where $f_N= \hat\Phi(f)$ and $f_L = f\circ f_N^{-1}$,
which proves that $\hat \Psi$ is a diffeomorphism.


Therefore, the map 
$\hat\Psi\colon\mathcal V_L\times \Sigma \to \mathcal V$ can be regarded as a local 
chart of $\AF$ and the map $\hat \Phi\colon\mathcal V \to \Sigma$ as a submersion defining
the foliation $\F_{\mathcal D}$ in a neighborhood of the identity.

Let $f\in \operatorname{Aut}_0(\F)$ be given. As each connected topological group, 
$\operatorname{Aut}_0(\F)$ is generated by a given 
neighborhood of the identity, in particular by $\mathcal V\cong \mathcal V_L\times \tilde\Sigma$.
It follows, using that $\D_{L}$ is a normal subgroup of $\operatorname{Aut}_0(\F)$, that
we can write
$$
f= f_L \circ f_N = f_L \circ\exp(v_1) \circ \dots \circ \exp(v_k)
$$
where $f_L\in \D_{L}$ and $v_i\in \X_{N,b}$ are vector fields with the property that, if we denote 
$\hat v_i = \vartheta(v_i) \in \G$, where $\vartheta\colon \X_{N,b}\to \G$ is the natural identification,
then $\exp(\hat v_i)\in \Sigma$. We denote by $\hat f_N$ the element of $G$ given by
$$
\hat f_N = \exp(\hat v_1) \circ \dots \circ \exp(\hat v_k)
$$ 
Now, proceeding as before we can prove that there exist neighborhoods 
$\mathcal V_f\subset \AF$, 
$\mathcal V_{L,f} \subset \D_{L}$ and $\Sigma_f \subset G$,
of $f$, $f_L$ and $\hat f_N$ respectively, and a commutative diagram
$$
\xymatrix{
\mathcal V_{L,f}\times \Sigma_f \ar[r]^{\hat\Psi_f} \ar[d]_{pr_2} & \mathcal V_f \ar[dl]^{\hat\Phi_f} \\ 
\Sigma_f & 
}
$$
where $\hat\Psi_f$ is a diffeomorphism providing a local chart of $\AF$
and $\hat\Phi_f$ is a smooth submersion defining $\mathcal F_{\mathcal D}$ in $\mathcal V_f$.
Finally, it also follows from the above discussion that, if $\mathcal V_{f_i}$ and $\mathcal V_{f_j}$
have non-empty intersection, then there are locally constant $G$-valued functions
$\gamma_{ij}$
fulfilling
$$
\hat \Phi_{f_i} = L_{\gamma_{ij}} \circ \hat \Phi_{f_j}.
$$ 
This ends the proof of the theorem
\end{proof}

\section{The automorphism group of a Riemannian foliation}
\label{RF}



We suppose in this section that the foliation $\F$ on $M$ is Riemannian. This means that the 
local submersions $\phi_i$ defining the foliation take values in a Riemannian manifold
$T$ and that the transformations $\gamma_{ij}$ fulfilling the cocycle condition \eqref{cocycle}
are local isometries of $T$. In an equivalent way, the foliation $\F$ is Riemannian if there is a 
Riemannian metric $g$ on $M$ which is bundle-like with respect to $\F$, that is a metric that 
in local adapted coordinates
$(x,y)$  is written
$$
g = \sum g_{ij} (x,y) \omega^i \omega^j 
+ \sum g_{a b} (y) dy^a dy^b,
$$ 
where $\{\omega^i, dy^a\}$ is a local basis of $1$-forms such that $\omega^i$ vanish on the 
bundle of vectors 
orthogonal to $\F$.

We suppose that such a 
bundle-like metric $g$ has been fixed. Let $p$ and $q$ stand for the dimension and the codimension
of $\F$ respectively.

In this section $\AF$ will denote the automorphism group of the Riemannian foliation $\F$,
that is the group of elements of $\DF$ which preserve the transverse Riemannian metric, and by 
$\aF$ the Lie algebra of vector fields whose flows are one-parameter 
subgroups of $\AF$. Also in this case, $\D_{L}$ is a normal subgroup of $\AF$ and $\X_{L}$ 
is an ideal of $\aF$. 
We denote by $\G = \aF/\X_{L}$ the quotient Lie algebra. The elements of $\G$ generate
local isometries on $T$, therefore they
are called basic Killing vector fields.

\begin{proposition}
The Lie algebra $\G$ of basic Killing vector fields has finite dimension.
\end{proposition}

The above proposition can be proved in the same way as the classical theorem of Myers and Steenrod
(cf.  \cite{Kob, Mol}). Namely, let $\pi\colon P\to M$ be the $O(q)$-principal fibre bundle 
of transverse orthonormal
frames of the Riemannian foliation $\F$. The total space $P$ is endowed with a foliation $\F_P$ 
of dimension $p$ that is projected by $\pi$ to $\F$. The foliation $\F_P$ is transversely
parallelizable and the elements of $\AF$ are in one-to-one correspondence with the 
automorphisms of the principal fibre bundle which preserve the transverse parallelism and the 
Riemannian connection. The finiteness of $\G$ then follows from the general fact according to which the 
automorphism group of a parallelism is finite dimensional.

In particular we can consider the simply connected Lie group associated to the Lie algebra $\G$
that we denote $G$. 

As we already mentioned Leslie proved in~\cite{Les2} that, in the case of a Riemannian foliation, 
the group  $\DF$ is a Fr\'{e}chet Lie group. 
Moreover, using the connection provided by the Riemannian metric, Omori's proof that $\D$ is a strong 
ILH-Lie group can be adapted to show the following

\begin{proposition}
Let $\F$ be a Riemannian foliation on a compact manifold $M$. Then $\DF$ is an
ILH-Lie group with Lie algebra $\mathfrak X_{\F}$.
\end{proposition}

Our main result in this section is the following theorem.


\begin{theorem}\label{LieGroup_R}
Let $\F$ be a Riemannian foliation on a compact manifold $M$.
The automorphism group $\AF$ of $\F$ is closed in $\D_{\F}$ and,
with the induced topology, it is a strong ILH-Lie group with Lie algebra $\aF$. 
Moreover, the left cosets of the subgroup $\D_{L}$  
define a Lie foliation $\F_{\mathcal D}$ on $\AF$, which is transversely modeled on 
the simply-connected complex Lie group $G$
associated to the Lie algebra $\G$.
\end{theorem}

\begin{remark}
The above theorem is weaker than Theorem~\ref{LieGroup_1} as we are not able to 
exhibit a supplementary to the Lie subalgebra $\aF$ inside the Lie algebra $\mathfrak X$
of $\D$.
Consequently we do not show that
$\AF$ is a strong ILH-Lie subgroup of $\D$.
Notice however that that Theorem \ref{LieGroup_R} implies in particular that
the topology of $\AF$ is second countable and LPSAC. 
\end{remark}

The proof of the above result is parallel to that of Theorem~\ref{LieGroup_1}, 
hence we just indicate the differences. 

\begin{proof}[Sketch of proof]
We recall first that the correspondence $f\mapsto \bar{f} = \Phi(f)$ given by
Proposition~\ref{P1} is well defined for each $f\in \D_{\F}$ close enough to the identity
(cf. Remark~\ref{D_F}). The proof of that proposition shows that there is a commutative 
diagram of continuous maps
\begin{equation}\label{cdiag_r}
\xymatrix{
\mathcal V_L\times \tilde\Sigma \ar[r]^{\Psi} \ar[d]_{pr_2} & \mathcal V \ar[dl]^{\Phi} \\ 
\tilde\Sigma & 
}
\end{equation}
where $\mathcal V$ and $\mathcal V_L$ are neighborhoods of the identity in $\AF$
and $\D_{\F}$ respectively, $\tilde\Sigma$ is the image of $\Phi$ and the map
$\Psi$, that is given by multiplication, is a homeomorphism.
As stated in Remark~\ref{cdiag_k} the same proof also provides commutative diagrams of continuous maps
between the Sobolev completions of the spaces in~\eqref{cdiag_r}.

Proceeding as in Proposition~\ref{P2}, we consider a complete transversal $T$ to $\F$
which now is endowed with a Riemannian metric. The families of local isometries of $T$ that 
are close to the identity and that
commute with the holonomy pseudogroup of $\F$ are the elements of a local group 
$\mathcal S$ that can be identified to $\tilde\Sigma$. 
We claim that $\mathcal S$ is locally compact. In the present setting this fact follows 
from the following observation that replaces Lemma~\ref{montel}. A local isometry is determined, on a 
connected domain, by its $1$-jet at a given point, therefore $\mathcal S$ can be embedded in a
space of $1$-jets over $T$, which is finite dimensional. Moreover, $\mathcal S$ is locally closed 
in this space of jets and therefore it is locally compact.

The arguments used by Salem in~\cite{Sal} can be applied here to show that $\mathcal S$
is in fact a local Lie group, more precisely each element of $\mathcal S$ is in the flow of a Killing 
vector field over $T$ that commutes with the holonomy pseudogroup of $\F$. These Killing vector fields 
form a Lie algebra naturally isomorphic to $\mathcal G$ and therefore $\tilde \Sigma$ is identified, 
through a suitable map $\sigma$, to 
a neighborhood $\Sigma$ of the neutral element of the simply connected Lie group $G$.
Notice that isometries of a smooth Riemannian metric are necessarily of class $C^\infty$. 
Therefore, and as it happens in Remark~\ref{sigma_k}, the map $\sigma$ also identifies 
$\Sigma$ with the set $\tilde\Sigma^k= \Phi(\mathcal V^k)$ for $k$ big enough.

The maps 
$\hat\Psi\colon\mathcal V_L\times \Sigma \to \mathcal V$ and 
$\hat\Psi\colon\mathcal V^k_L\times \Sigma \to \mathcal V^k$
can be regarded as local 
charts of $\AF$ and $\operatorname{Aut}^k(\F)$ respectively. 
Moreover the map $\hat \Phi\colon\mathcal V \to \Sigma$ is a submersion defining
the foliation $\F_{\mathcal D}$ in a neighborhood of the identity. 
In a similar way as in the proof of Theorem~\ref{LieGroup_1}, one can construct 
atlases of adapted local charts for $\AF$ and $\operatorname{Aut}^k(\F)$. In particular 
the groups $\operatorname{Aut}^k(\F)$ turn out to be Hilbert manifolds showing that $\AF$
is an ILH-Lie group. This ends the proof.
\end{proof}

%
%

\end{document}